\documentclass{amsart}
\usepackage{euler,amsmath,amssymb,amscd,amsfonts}
\usepackage{epsfig, graphicx,pinlabel,pdfsync}
\title{Note on the decomposition of states}
\date{\today}
\author{Donghoon Hyeon}
\author{Jaekwang Kim}
\address[DH]{Department of Mathematical Sciences, Seoul National University, Seoul, 151-747, R. O. Korea \\ Tel: +82-2-880-2666, Fax: +82-2-887-4694  }
\email{dhyeon@snu.ac.kr}
\address[JK]{Department of Mathematics, POSTECH, Pohang, Gyungbuk 790-784, R. O. Korea}

\newtheorem{theorem}{Theorem}[section]

\newtheorem{lemma}[theorem]{Lemma}
\newtheorem{prop}[theorem]{Proposition}
\newtheorem{coro}[theorem]{Corollary}
\theoremstyle{definition}

\def\mbb{\mathbb}

\def\tu{\textup}

\def\a{\alpha}

\def\la{\lambda}
\def\m{\mu}

\def\bR{\mbb R}
\def\bP{\mbb P}
\def\bZ{\mbb Z}

\def\SL{{\mathrm SL}}
\def\GL{{\mathrm GL}}

\def\bP{\mathbb P}
\def\bG{\mathbb G}

\def\inj{\hookrightarrow}

\def\spec{\tu{Spec \,}}

\def\hilb{\textup{Hilb}}

\def\inj{\hookrightarrow}

\def\Sym{\mathrm{Sym}}

\input xy
\xyoption{all}

\input epsf
\def\ii{{\bf in}\,}
\def\cP{\mathcal P}
\def\st{{\mathcal P}}

\newtheorem{remark}[theorem]{Remark}

\begin{document}

\begin{abstract} We derive a sharp decomposition formula for the state polytope of the Hilbert point and the Hilbert-Mumford index of reducible varieties by using the decomposition of characters and basic convex geometry. This proof captures the essence of the decomposition of the state polytopes in general, and considerably simplifies an earlier proof by the authors which uses a careful analysis of initial ideals of reducible varieties.
\end{abstract}


\thanks{The first named author was supported by the Research Resettlement Fund for the new faculty of Seoul National University, and the following grants funded by the government of Korea:
NRF grant 2011-0030044 (SRC-GAIA) and NRF grant NRF-2013R1A1A2010649.}
\maketitle

\section{Introduction \& Preliminaries} \label{S:intro}
In this article, we take a new look at the decomposition formula for state polytopes \cite{HyK} from a more general point of view. We shall work over an algebraically closed field $k$ of characteristic zero. Let $G$ be a linear algebraic group and  $R$ be a maximal torus of it. Let $W$   be a rational representation of $G$ and $w \in W$ be a point. Recall that the {\it state} $\Xi_w(R)$ of $w$ with respect to $R$ is the set of the characters $\chi \in X(R)$ such that $w_\chi \ne 0$. Since $w = \sum_\chi w_\chi$ implies $c w = \sum_\chi c w_\chi$, we have $\Xi_w(R) = \Xi_{c w}(R)$ for any nonzero $c \in k$. Hence we may define the state $\Xi_{p}(R)$ of $p \in \bP(W)$ to be $\Xi_{w}(R)$ for any affine point $w \in W$ over $p$. (We conflate a vector space $W$ with the affine scheme $\spec \Sym(W^*)$.)

We shall be concerned with  the states of Hilbert points of homogeneous ideals.
Let $S = k[x_0,\dots,x_n]$ and let $P(u)$ be a rational polynomial in one variable $u$ and $Q(u) = \binom{u+m}m - P(u)$. If $m$ is at least the Gotzmann number of $P(u)$, then for any saturated homogeneous polynomial $I \subset k[x_0,\dots, x_n]$ whose Hilbert polynomial is $P$, the {\it $m$th Hilbert point}
$[I]_m$ of $I$ is well defined as a $Q(m)$-dimensional subspace  of the degree $m$ piece $S_m$:
 \[
 [I]_m : [I_m \inj S_m] \in \textup{Gr}_{Q(m)}S_m \inj \bP\left(\bigwedge^{Q(m)} S_m \right).
 \]
 The {\it dual $m$th Hilbert point} $[I]_m^*$ is defined by
 \[
 [I]_m^* : [S_m \to S_m/I_m] \in \textup{Gr}^{P(m)}S_m \inj \bP\left(\bigwedge^{P(m)} S_m^*\right).
 \]
 The collection of $m$th Hilbert points form a closed subscheme $\hilb_P\bP(V)$, called the Hilbert scheme, of the Grassmannian. Working out the geometric invariant theory (GIT) of suitable Hilbert schemes give rise to various moduli spaces \cite{GIT}, and our main motivation for the study in this article is the construction of the moduli of curves. The link between GIT and the study of state polytopes is given by the following fundamental observation (numerical criterion): \emph{If $G$ is reductive and $V$ is a rational representation, $v \in V$ is GIT unstable if and only if there is a torus $R$ of $G$ such that the convex hull of $\Xi_v(R)$ does not contain the trivial character. }

 The {\it monomial basis} of $\bigwedge^{Q(m)}S_m$ consists of  $x^\a:=x^{\a(1)}\wedge\cdots\wedge x^{\a(Q(m))}$, the wedge product of $Q(m)$  degree $m$ monomials $x^{\a(i)}$'s. The basis members are also the $R$-weight vectors of $\bigwedge^{Q(m)}S_m$, where $R$ is the maximal torus of $G = \GL(S_1)$ diagonalized by $x_0, \dots, x_n$: Indeed, let $\chi_i$ be the character of $R$ determined by $
 t.x_i=  \chi_i(t)x_i$. Then by letting $t_i$ denote $\chi_i(t)$ and using the usual multivector notation $t^{\gamma} = \prod_{i=0}^n t_i^{\gamma_i}$, $t \in T$, $\gamma \in \bZ^{n+1}$,  we have $t.x^\a = t^{\sum_{i=1}^{Q(m)}\a(i)} x^{\a}$, which means precisely that
 $
\left(\bigwedge^{Q(m)}S_m\right)_{\chi^\a} = k. x^{\alpha}
 $
 where $\chi^\a = \prod_{i=1}^{Q(m)}\prod_{j=0}^n \chi_j^{\a(i)_j}$.

 The Hilbert point $[I]_m$ has a nonzero $x^\a$-coefficient if and only if the degree $m$ monomials other than $x^{\a(1)}, \dots, x^{\a(Q(m))}$ form a $k$-basis of $S_m/I_m$.
 Following \cite{Kempf}, we denote the set of states by $\Xi_{[I]_m}(R)$ and  its convex hull by $\cP_m(I)$. We call $\cP_m(I)$ the {\it $m$th state polytope} of $I$, following \cite{BM}.

\

For any fixed sufficiently large $m$, Bayer and Morrison proved that
\begin{theorem}\label{T:BM}\cite[Theorem~3.1]{BM} There is a canonical bijection between the vertices of $\cP_m(I)$ and the initial ideals $\ii_\prec(I)$ as $\prec$ runs through all term orders on $k[x_0,\dots,x_n]$.
\end{theorem}
Using the Bayer-Morrison theorem and basic properties of monomial orders and initial ideals,
decomposition formulae for initial ideals, state polytopes,  and Hilbert-Mumford indices were achieved in \cite{HyK}:

\begin{theorem}\cite{HyK}\label{T:st-decomposition} Let $X$ be a chain of projective varieties $X_1, \dots, X_\ell$ defined by a saturated homogeneous ideal $I_X = \cap_i I_{X_i}$ i.e. $X = \cup_{i=1}^\ell X_i$ and $X_i$ meets $X_j$ when and only when $|i-j|=1$. Suppose that there is a homogeneous coordinate system $x_0,\dots,x_n$ and a sequence $n_{0}=0 < n_1 < \cdots < n_\ell = n$ such that
\[
X_i \subset \{x_0 = \cdots = x_{n_{i-1}-1} = 0, x_{n_i+1} = x_{n_i+2} = \cdots = x_n = 0\}.
\]
Then the state polytope of $X$ is given by the following decomposition formula
\small\begin{equation}\label{E:main}
\cP_m(I_X) = \sum_{i=1}^\ell \cP_m( I_{X_i}\cap k[x_{n_{i-1}},\cdots,x_{n_i}]) + \sum_{i=1}^{\ell-1} \cP_m(T_i\cap k[x_{n_{i-1}}, \dots, x_n])
\end{equation} \normalsize
where  $T_i = \langle x_{n_{i-2}},\dots,x_{n_{i-1}-1}\rangle \langle x_{n_i+1},\dots,x_n \rangle$ for $2 \le i \le \ell-1$, and $T_1 = \langle x_{n_1+1},x_{n_1+2},\dots,x_n\rangle$ and $T_{\ell} = \langle x_{n_{\ell-2}},x_{n_{\ell-2}+1},\dots,x_{n_{\ell-1}-1}\rangle$.
\end{theorem}
Here, $\st_m( I_{X_i}\cap k[x_{n_{i-1}},\cdots,x_{n_i}])$ is regarded as a convex polytope in the subspace
$
\{{\bf a} \in \bR^{n+1} \, | \, a_i = 0, \, \forall i < n_{i-1}, \forall i>n_i \}.
$
Similarly, $\st_m(T_i\cap k[x_{n_{i-1}}, \dots, x_n])$ is also regarded as a convex polytope in the relevant vector subspace. Note that $\cP_m(T_i\cap k[x_{n_{i-1}}, \dots, x_n])$ is a point since $T_i$ is a monomial ideal. We let $\tau$ denote the point $\sum_{i=1}^{\ell-1} \cP_m(T_i\cap k[x_{n_{i-1}}, \dots, x_n])$.
We shall show in Section~\ref{S:proof-main} that Theorem~\ref{T:st-decomposition} is a direct consequence of the  fact that the characters of a direct sum is the sum of the characters (Proposition~\ref{P:st-decomposition}).

It is also shown in \cite{HyK} that this decomposition is sharp: the vertices of $\cP_m(I_X)$ are precisely the sums of vertices of $\cP_m( I_{X_i}\cap k[x_{n_{i-1}},\cdots,x_{n_i}])$ and $\tau$ (Corollary~\ref{C:vertex}).
The proof in \cite{HyK} uses Theorem~\ref{T:BM} and the initial ideal decomposition formula.
   We shall show in Section~\ref{S:vertex} that the sharpness of the decomposition is in fact a consequence of a general convex geometry phenomenon.

Finally, we also reprove in Section~\ref{S:HMindex} the Hilbert-Mumford index decomposition formula below  by  using the decomposition of characters.
\begin{prop}\cite{HyK}\label{P:mu-decomposition} Let $X$ be as in Theorem~\ref{T:st-decomposition} and $\rho:\mathbb{G}_m\to \GL_{n+1}$ be a 1-parameter subgroup of $\GL_{n+1}$ diagonalized by $\{x_0,\dots,x_n\}$ with weights $(r_0,\cdots,r_n)$ and $\rho_i$ be the restriction of $\rho$ to $\GL(kx_{n_{i-1}}+\cdots+kx_{n_i})$.  Then the Hilbert-Mumford index $\mu([X]^*_m,\rho)$ of the (dual) $m$th Hilbert point of $X$ with respect to $\rho$ is given by
\small
\[
\mu([X]^*_m,\rho)=\sum_{i=1}^{\ell}\mu([X_i]^*_m,\rho_i)-\sum_{i=1}^{\ell}\bigg(\frac{mP_i(m)}{n_{i}-n_{i-1}+1}\sum_{k=n_{i-1}}^{n_i}r_k\bigg)
+ \frac{m P(m)}{n+1}\sum_{i=0}^n r_i +m\sum_{i=1}^{\ell-1}r_{n_i}
\]
\normalsize
where $P(m)$ is the Hilbert polynomial of $I_X \subset k[x_0,\dots,x_n]$ and  $P_i(m)$, the Hilbert polynomial of $I_{X_i} \cap k[x_{n_{i-1}}, \dots, x_{n_i}]$ regarded as an ideal in $k[x_{n_{i-1}},\dots,x_{n_i}]$.
\end{prop}

We close this section with an observation that will be used in Section~\ref{S:HMindex}: We shall prove that the Hilbert-Mumford index of $[I_X]_m$ and that of the dual Hilbert point $[I_X]_m^*$ are  the same. Let $\rho$ be a one-parameter subgroup of $SL(S_1)$. We recall the fact that $\lim_{t\to 0}\rho(t).[I_X]_m = [\ii_{\prec_\rho}I_X]_m$ where $\prec_\rho$ is the $\rho$-weight order with the reverse lexicographic tie-breaking \cite{BM}. Then the Hilbert-Mumford index is
\[
\mu([I_X]_m, \rho) =  \mu([\ii_{\prec_\rho}I_X]_m, \rho) = - \sum_{x^\a \in (\ii_{\prec_\rho}I_X)_m} \textrm{wt}_\rho(x^\alpha).
\]
Let $\{f_0, \dots, f_n\} \subset S_1^*$ be the dual basis of $\{x_0, \dots, x_n\} \subset S_1$. Use the multi-vector notation $f^\a = \prod_{i=0}^nf_i^{\a_i}$. Then $f^{\a(1)}\wedge \cdots f^{\a(P(m))}$ appears in $\bigwedge^{P(m)} (S/I_X)_m^*$ with a nonzero Pl\"ucker co-ordinate if and only if $x^{\a(1)}, \dots, x^{\a(P(m))}$ form a basis of $(S/I_X)_m$. Since $\rho$ is a co-character of the special linear group, the weights of all monomials of $S_m$ sum up to zero. Also, $\rho$ acts on $\bP(\bigwedge S_m)$ and $\bP(\bigwedge S_m^*)$ with opposite weights. Hence we have
\[
\begin{array}{ccl}
\mu([I_X]_m^*, \rho) & = & \max\{- \sum_{x^\a \in \mathcal B}\textrm{wt}_\rho(f^\a) \, | \, \mbox{$\mathcal B$ a monomial basis of $(S/I_X)_m$}\}\\
& = & - \sum_{x^\a \in S_m\setminus \ii_{\prec_\rho}(I_X)} \left( - \mathrm{wt}_\rho(x^\a)\right) = -  \sum_{x^\a \in (\ii_{\prec_\rho}(I_X))_m} \mathrm{wt}_\rho(x^\a).
\end{array}
\]

\section{Decomposition of states}
Let $G = \GL(V)$ and let $V_i$, $i = 1, \dots, \nu$, be vector subspaces of $V$ that span $V$. Note that $V = \sum_{i=1}^\nu$ is {\it not necessarily a direct sum}. Then
\[
S^mV = \sum_{i=1}^\nu S^m V_i + \sum_{\substack{ \sum_{j}m_{j}=m\\ 0 < m_j < m}} \bigotimes_{j=1}^\nu S^{m_j} V_j.
\]
For a notational convenience, we let $W$ denote $S^mV$, $W_j = S^mV_j$ for $1 \le j \le \nu$ and
\[
W_{\nu+1} =  \sum_{\substack{ \sum_{j}m_{j}=m\\ 0 < m_j < m}} \bigotimes_{j=1}^\nu S^{m_j} V_j.
\]

Let $R$ be a maximal torus of $\GL(V)$ which preserves the subspaces $V_i$. Then one can choose a basis of $\mathcal B = \{v_1, \dots, v_M\}$ of $V$ diagonalizing the $R$-action such that $V_j$ is the linear subspace spanned by $\{v_{M_j}, v_{M_j+1}, \dots, v_{M_j'}\}$.  We identify $GL(V_j)$ with the subgroup of $GL(V)$ which preserves $V_j$ and acts trivially on $\mathrm{Span}\{v_i \, |  i < M_j, i > M_j'\}$.

Let $\chi_i$ be the character of $R$ determined by $t.v_i = \chi_i(t)v_i$.  Set $R_j = R \cap \GL(V_j)$. Then there is a  natural
projection $\pi_j : R \to R_j$ defined by
\[
\chi_s(\pi_j(t)) = \begin{cases} \chi_s(t), & M_j \le s \le M'_j \\ 1, & \mbox{else}\end{cases}.
\]
Then $\pi_j$'s induce injective group homomorphisms $\pi_j^*: X(R_j) \inj X(R)$. We shall identify $X(R_j)$ with its image in $X(R)$ under $\pi_j^*$.

\begin{prop}\label{P:st-decomposition} Let $I$ be a subspace of $W$, $I_j = I\cap W_j$, and suppose that the sum
$
I = \sum_{j=1}^{\nu+1} I_j
$ is direct. Let $\dim I = N$ and $\dim I_j = N_j$, $1 \le j \le \nu+1$. We have the decomposition of states
\[
\Xi_{\left[\bigwedge^N I\right]}(R) = \sum_{j=1}^\nu \Xi_{\left[\bigwedge^{N_j}I_j\right]}(R_j) + \Xi_{\left[\bigwedge^{N_{\nu+1}}I_{\nu+1}\right]}(R).
\]
\end{prop}


\begin{proof}

Let $\xi$ be an affine point over $\left[\bigwedge^NI\right] \in \bP(\bigwedge^NS^mV)$, and let $\xi_j$ be an affine point over $\left[\bigwedge^{N_j}I_j\right]$ in $\bigwedge^{N_j}W_j$.
The affine point $\xi_j$ generates the one-dimensional subspace $\bigwedge^{N_j} I_j$. Let $j \le \nu$. Consider the $R$-weight decomposition of $\xi_j$.  Since $I_j$ is contained in the $R$-module $W_j = S^mV_j$, the $R$-weight decomposition of $\xi_j$ is precisely the $R_j$-weight decomposition i.e.
 \[
 \xi_j = \sum_{\chi \in X(R)} (\xi_j)_{\chi} = \sum_{\chi \in X(R_j)} (\xi_j)_{\chi}.
 \]
Since the sum $I = \sum_{j=1}^{\nu+1} I_j$ is direct, we have
\[
\bigwedge^N I \simeq \bigotimes_{j=1}^{\nu+1}\left(\bigwedge^{N_j}I_j \right)
\]
and hence the $R$-weight decomposition of $\xi$ is given as
\begin{equation}\label{E:states}
\xi = \sum_{\substack{\chi^{(j)} \in X(R_j) \\ \chi^{\nu+1}\in X(R)}}  \bigotimes_{j=1}^{\nu+1}(\xi_j)_{\chi^{(j)}}
\end{equation}
A summand $\bigotimes_{j=1}^{\nu+1}(\xi_j)_{\chi^{(j)}}$ has a weight $\sum_{j=1}^{\nu+1}\chi^{(j)} \in X(R)$ and it is a state of $\xi$ if and only if the weight vector $\sum_{j=1}^{\nu+1} (\xi_j)_{\chi^{(j)}}$ is nonzero if and only if $\chi^{(j)}$ is a state of $\xi_j$, $\forall j$. It follows that every state of $\xi$ is a sum of states of $\xi_i$'s and vice versa.
\end{proof}

\subsection{Proof of Theorem~\ref{T:st-decomposition}}\label{S:proof-main}
We shall now deduce Theorem~\ref{T:st-decomposition} from Proposition~\ref{P:st-decomposition}. Let $X = X_1 \cup X_2 \subset \bP(V^*)$ be a chain of subvarieties $X_i$ and suppose that there exists a homogeneous coordinate system $x_0,\dots, x_\ell, \dots, x_n \in V$ such that
\[
\begin{array}{l}
	X_1 \subset \{x_{\ell+1} = \cdots = x_n = 0\} \\
	X_2 \subset \{x_0 = \cdots = x_{\ell-1} = 0\}.
\end{array}\tag{$\dagger$}
\]
We also assume that $X_1\cap X_2 \ne \emptyset$, and that $X, X_1, X_2$ are cut out by saturated homogeneous ideals $I_{X_i}$ and $I_X = I_{X_1}\cap I_{X_2}$.

Let $V_1$ (resp. $V_2$) be the subspace of $V$  spanned by $\{x_0, \dots, x_\ell\}$ (resp. $\{x_\ell, \dots, x_n\}$). Let $W = S^mV$, $W_i = S^mV_i$ for $i = 1,2$, and $W_3 = \sum_{i+j=m, ij\ne 0}S^iV'_1 \otimes S^jV'_2$ where $V'_1 = \sum_{i=0}^{\ell-1} k x_i$ and $V'_2 = \sum_{i=\ell+1}^n k x_i$.
Evidently we have $W = \sum_{i=1}^3 W_i$. Let $R$ be the maximal torus of $\GL(V)$ diagonalized by $x_0,\dots,x_n$ and $R_i = R \cap \GL(V_i)$ for each $i$, where $\GL(V_i)$ is identified with a suitable subgroup of $\GL(V)$ as in the discussion preceding Proposition~\ref{P:st-decomposition}. Of course, $R_1$ (resp. $R_2$) is identified with the maximal torus of $\GL(V_1)$ (resp. $\GL(V_2)$) diagonalized by $x_0,\dots, x_\ell$ (resp. $x_\ell, \dots, x_n$).

For each $m\ge 2$, we have
\begin{equation}\label{E:decomp-IX}
(I_X)_m = (I_X\cap W_1) + (I_X\cap W_2) + (I_X \cap W_3).
\end{equation}
We claim that the property of coordinates ($\dagger$) implies that this is a direct sum decomposition.
Indeed, since $W_1\cap W_2$ is the $1$-dimensional space $k.x_\ell^m$ and $x_\ell^m$ does not vanish at $\{p\} = X_1\cap X_2$, $I_X\cap W_1 \cap W_2 = 0$. Since $W_1$ and $W_2$ meet $W_3$ trivially, the claim follows and we may apply Proposition~\ref{P:st-decomposition}.

Note that the three terms on the right hand side of (\ref{E:decomp-IX}) are $I_{X_1}\cap k[x_0,\dots, x_\ell]_m$, $I_{X_2} \cap k[x_\ell,\dots, x_n]_m$ and $T_m$ respectively, where $T = \langle x_0,\dots, x_{\ell-1}\rangle \langle x_{\ell+1},\dots,x_n\rangle$. Then by Proposition~\ref{P:st-decomposition}, we have
\begin{equation}\label{E:states-hilbert}
\Xi_{\bigwedge^N (I_X)_m}(R) = \Xi_{\left[\bigwedge^{N_1}(I_X\cap W_1)\right]}(R_1) + \Xi_{\left[\bigwedge^{N_2}(I_{X}\cap W_2)\right]}(R_2) + \Xi_{\left[\bigwedge^{N_3}T_m\right]}(R) \tag{$\dagger\dagger$}
\end{equation}
where $N$ and $N_i$ denote the appropriate dimensions. As observed in the introduction, we may naturally identify the characters $\prod_{i=0}^n \chi_i^{a_i}$ with the Laurent monomials $\prod_{i=0}^n x_i^{a_i} \in k(x_0,\dots,x_n)$, where $\chi_i$ is the projection $t=(t_0, \dots, t_N) \mapsto t_i$. Identifying the characters with monomials and taking the convex hull of both sides, we obtain Theorem~\ref{T:st-decomposition} for the case $\ell = 2$ from which, as observed in \cite{HyK}, the general case follows by a simple induction.

\begin{remark} Note that since $T$ is a monomial ideal, $\Xi_{\left[\bigwedge^{N_3}T_m\right]}(R)$ consists of one point $\chi^\tau$ where $\tau = \sum_{\substack{x^\a \in T_m}} \a$.
	\end{remark}

\subsection{Decomposition of vertices}\label{S:vertex}
Let $P_1, \dots, P_r$ be polytopes in $\bR^n$. In general, every {\it face} $F$ of the Minkowski sum $\sum_{i=1}^r P_i$ has a unique decomposition $F = \sum_{i=1}^r F_i$ into a sum of faces $F_i$ of $P_i$. The converse is easily seen to be false: If the origin ${\bf 0}$ is a vertex of a polytope $P$, then ${\bf 0} + v = v$ is not a vertex of $2P$ for any nonzero vertex $v$ of $P$. The following lemma guarantees that vertices always sum up to be a vertex provided that the polytopes are positioned well enough.

\begin{lemma}\label{L:vertex} Let $P_1, P_2$ be polytopes in $\bR^n$. Suppose that $P_1$ and $P_2$ are contained in affine hyperplanes $H_1$ and $H_2$ respectively such that $H_1\cap H_2$ is of dimension one. Then the vertices of the Minkowski sum $P_1 + P_2$ are precisely the sums of vertices of $P_1$ and $P_2$.
\end{lemma}

\begin{proof}It is evident that a vertex of $P_1+P_2$ is a sum of vertices of $P_1$ and $P_2$ since for any subsets $S_1$ and $S_2$ of $\bR^n$, the sum of their convex hulls is the convex hull of their sum $S_1 + S_2$. To prove the converse, we start by choosing affine coordinates $x_1, \dots, x_n$ judiciously so that
\[
H_1 = \left\{ {\bf x}\in \bR^n \, \left| \, \sum_{i=1}^{n_1} x_i = N_1, x_i = 0 \, \forall i > n_1\right\} \right.
\]
and
\[
H_2 = \left\{ {\bf x} \in \bR^n \, \left| \, \sum_{i=n_1}^n x_i = N_2, x_i = 0 \, \forall i < n_1\right\}.\right.
\]

Let $\{v_1, \dots, v_r\}$ and $\{w_1,\dots, w_s\}$ be the sets of vertices of $P_1$ and $P_2$, respectively. We aim to show that $v_i + w_j$ is a vertex of $P_1 + P_2$, for any $i, j$. Suppose it is not the case - suppose without losing generality $v_1 + w_1$ is not a vertex. Then there
 exist $\la_{ij}$ for $i=1,\dots,r, j=1,\dots,s$ such that $\sum_{i,j}\la_{ij}=1$, $0\leq\la_{ij}\leq 1$, $\la_{11}=0$ and
\[
v_1+w_1=\sum_{i,j}\la_{ij}(v_i+w_j).
\]
By rearranging the terms, we have
\[
\left(1-\sum_{j}\la_{1j}\right)v_1-\sum_{i\neq 1}\la_{ij}v_i=\sum_{j\neq 1}\la_{ij}w_j+\left(\sum_{i}\la_{i1}-1\right)w_1 \in H_1 \cap H_2
\]
which implies that $x_{n_1}$ is the only nonzero coordinate of each side. Moreover, $1-\sum_{j}\la_{1j}\neq 0$ or $\sum_{i}\la_{i1}-1\neq 0$ since $\sum_{i,j}\la_{ij}=1$. Suppose the $1-\sum_{j}\la_{1j}\neq 0$ (the other case is proved similarly) and let $v_i = (v_{i0},v_{i1},\dots,v_{in_1})$. Then we have
\[
v_{1k}=\sum_{\substack{i\neq 1\\1\le j \le s}}\m_{ij}v_{ik} \quad k\ne n_1
\]
where $\m_{ij}=\la_{ij}/(1-\sum_{j'=1}^s\la_{1j'})$ and $\sum_{\substack{i\neq 1\\1\le j \le s}}\m_{ij}=1$.

The $n_1$th coordinate also satisfies the above condition because
\begin{align*}
v_{1n_1}&=N_1-\sum_{k=0}^{n_1-1}v_{1k}= N_1-\sum_{k=0}^{n_1-1}\sum_{i\neq 1}\m_{ij}v_{ik}
=\sum_{i\neq 1}\m_{ij}N_1-\sum_{i\neq 1}\sum_{k=0}^{n_1-1}\m_{ij}v_{ik}\\
&=\sum_{i\neq 1}\m_{ij}\left(N_1-\sum_{k=0}^{n_1-1}v_{ik}\right)=\sum_{i\neq 1}\m_{ij}v_{in_1}
\end{align*}
which means that $v_1=\sum_{i\neq 1}\left(\sum_{j=1}^s\m_{ij}\right)v_i$.
But this is a contradiction since $v_1$ is a vertex of $P_1$.
\end{proof}

As an immediate corollary, we obtain:
\begin{coro}\label{C:vertex} Retain notations from Theorem~\ref{T:st-decomposition}. Let $\mathcal V_i$ denote the set of vertices of $\st_m(I_{X_i}\cap k[x_{n_{i-1}},\dots,x_{n_i}])$, $i = 1, \dots, \ell$. Then the vertices of $\st_m(I_X)$ are precisely
\[
\left\{ \left. \tau + \sum_{i=1}^\ell v_i \, \right| \, v_i \in \mathcal V_i \right\}.
\]
\end{coro}

\begin{proof} The $\ell = 2$ case follows from Lemma~\ref{L:vertex}: The state polytopes $\st_m(I_{X_1}\cap k[x_0, x_1, \dots,x_{n_1}])$ and $\st_m(I_{X_2}\cap k[x_{n_1},x_{n_1+1}\dots,x_n])$ are in the affine hyperplanes
\[
H_1 = \left\{ {\bf x} \in \bR^n \, | \, \sum_{i=0}^{n_1} x_i = Q_1(m), \, x_i = 0, \, \forall i > n_1\right\}
\]
and
\[
H_2 = \left\{ {\bf x} \in \bR^n \, | \, \sum_{i=n_1}^{n} x_i = Q_2(m), \, x_i = 0, \, \forall i < n_1\right\}
\]
respectively, where $Q_1(m) = \dim I_{X_1}\cap k[x_0, x_1, \dots,x_{n_1}]_m$ and $Q_2(m) = \dim I_{X_2}\cap k[x_{n_1},x_{n_1+1}\dots,x_n]_m$. Since $H_1 \cap H_2$ is one dimensional, Lemma~\ref{L:vertex} applies.
The general, $\ell \ge 2$ case follows by an induction.
\end{proof}

\section{Decomposition of Hilbert-Mumford index}\label{S:HMindex}

  Retain the notations from Section~\ref{S:proof-main}.
  To prove Proposition~\ref{P:mu-decomposition},  as in \cite{HyK} we shall assume that $\ell = 2$ as the general case follows by a simple induction.
  We let $Q(m) = \dim (I_X)_m$ and $P(m) = \dim (S/I_X)_m = \dim S_m - Q(m)$. Likewise, $Q_i(m) = \dim (I_{X_i})_m$ and $P_i(m) = \dim S^mV_i - Q_i(m)$.
   Let $\rho$
be a 1-PS of $R$ and let $\rho_i$ be the induced 1-PS of $R_i$, $i = 1,2$. These are obtained by composing with the projections $R \to R_i$.
Recall that, if the sum of the $\rho$-weights is zero,  the Hilbert-Mumford index is given by
\[
\mu([I_X]_m, \rho) = \max\{-\langle\chi, \rho\rangle\, | \, \chi\in \Xi_{[I_X]_m}(R) \}
\]
where $\langle \, \, , \, \, \rangle$ denotes the natural paring of the character group and the 1-PS group i.e. $\chi\circ \rho(t) = t^{\langle \chi, \rho \rangle}$ for any $t \in \bG_m(k)$.

For any $\chi\in \Xi_{[I_X]_m}(R)$, due to (\ref{E:states-hilbert}), we have $\chi = \chi_1 + \chi_2 + \tau$, $\chi_i = \iota_i \circ \chi$ where $\iota_i :  R_i \to R$ is the inclusion. And $\tau$ is the character with which $\rho$ acts on $\bigwedge^{\textrm{max}} \left(I \cap \sum_{i+j = m, ij\ne 0} S^iV_1\otimes S^jV_2\right) = T_m$ where $T = \langle x_0,\dots, x_{\ell-1}\rangle \langle x_{\ell+1},\dots,x_n\rangle$.

Hence we have
\[
\begin{array}{lcr}
\langle\chi, \rho\rangle & = & \langle\chi_1,\rho_1\rangle +  \langle\chi_2, \rho_2\rangle + \langle\tau, \rho\rangle.\\
\end{array}
\]
Clearly, the minimum of $\langle\chi, \rho\rangle$ is achieved precisely when each $\chi_i$ pairs minimally with $\rho$. Let $\rho'$ be the 1-ps of $\SL(V)$ associated to $\rho$ i.e. if $r_i$ are the weights of $\rho$, then $\rho'$ is the 1-ps with weights $r_i - w$ where $w$ is the average of the weight $\frac1{\dim V} \sum_{i=1}^n r_i$. Conflating a 1-ps with its weight vector, we may write $\rho = \rho' + (w,w,\dots,w)$.

The minimum of $\langle \chi, \rho \rangle$ is achieved by
\[
\sum_{x^\a \in \left(\ii_{\prec_\rho}I_X\right)_m} \chi^\a
\]
where we used the multiplicative multi-vector notation $\chi^\a = \prod \chi_i^{\a_i}$ as in the discussion preceding Theorem~\ref{T:BM}. Note that
\[
\begin{array}{cll}
\langle \chi, \rho \rangle & = & \langle \chi, \rho' \rangle + \langle \chi, (w,w,\dots,w) \rangle \\
& = & - \mu([I_X]_m, \rho) + w \sum_{x^\a \in \left(\ii_{\prec_\rho}I_X\right)_m}(\sum_{i=0}^n \a_i )\\
& = & - \mu([I_X]_m, \rho) + w m Q(m).
\end{array}
\]
Similarly, let $\rho_i'$ denote the 1-ps of $\SL(V_i)$ associated to $\rho_i$, $i = 1, 2$, whose weights are shifted by the  average weight $w_i = \frac{\sum_{x_j \in V_i} \mathrm{wt}_\rho(x_j)}{\dim V_i}$. Clearly, $\chi_i$ pairs with $\rho_i$ minimally if and only if it pairs with $\rho_i'$ minimally, and
\[
\min \langle \chi_i, \rho_i\rangle = -\mu([I_{X_i}\cap S^mV_i]_m, \rho_i) + w_i m Q_i(m).
\]
Hence we have
\begin{equation}
\begin{array}{cll}
\mu([I_X]_m, \rho) & = &-\min \langle \chi_1, \rho_1 \rangle - \min \langle \chi_2, \rho_2 \rangle - \langle \tau, \rho \rangle + mwQ(m) \\
& =  &\mu([I_{X_1}\cap k[x_0,\dots,x_\ell]]_m,\rho_1) + \mu([I_{X_2}\cap k[x_\ell,\dots,x_n]]_m,\rho_2) \\
&  & - w_1 m Q_1(m) - w_2 mQ_2(m) - \sum_{x^\a \in T_m}\mathrm{wt}_\rho(x^\a) + w m  Q(m).
\end{array}
\end{equation}
Substitute $\dim S^mV - Q(m) = P(m)$ and $\dim S^mV_i - P_i(m) = Q_i(m)$. And  subsequently, substitute   $\sum_{x^\a \in S^mV} \mathrm{wt}_\rho(x^\a)$ for $m w \dim S^mV$ and $\sum_{x^\a \in S^mV_i} \mathrm{wt}_\rho(x^\a)$ for $m w_i \dim S^mV_i$. Then we get
\begin{equation}\label{E:mu-decomp}
\begin{array}{cll}
\mu([I_X]_m, \rho) & = &\mu([I_{X_1}\cap k[x_0,\dots,x_\ell]]_m,\rho_1) + \mu([I_{X_2}\cap k[x_\ell,\dots,x_n]]_m,\rho_2) \\
&  & + w_1 m P_1(m) + w_2 m P_2(m) + \mathrm{wt}_\rho(x_\ell^m) - w m P(m)
\end{array}\tag{$\ddagger$}.
\end{equation}
since $S^mV_1 \coprod (S^mV_2 \setminus \{x_\ell^m\}) \coprod  T_m = S^m V$.   Recall from the  closing moment of the introduction that $[I_X]_m$ and $[I_X]_m^*$ have the same Hilbert-Mumford indices, and that $\rho$ acts with opposite weights on $\bP(\bigwedge S_m^*)$ in which the dual Hilbert points live. That is, if $w^*$ is the average of the weights for the $\rho$ action on $\bP(\bigwedge S_m^*)$, then $w* = -w$ so that the signs of the terms $w m P(m)$, $w_i m P_i(m)$ are reversed.    Hence (\ref{E:mu-decomp}) is precisely the assertion of Proposition~\ref{P:mu-decomposition} for the case $\ell = 1$.


\begin{thebibliography}{Mum65}

\bibitem[BM88]{BM}
David Bayer and Ian Morrison.
\newblock Standard bases and geometric invariant theory. {I}. {I}nitial ideals
  and state polytopes.
\newblock {\em J. Symbolic Comput.}, 6(2-3):209--217, 1988.
\newblock Computational aspects of commutative algebra.

\bibitem[HK]{HyK}
Donghoon Hyeon and Jaekwang Kim.
\newblock A state polytope decomposition formula.
\newblock arXiv:1304.0218 [math.AG], to appear in the {\it Proceedings of the
  Edinburgh Mathematical Society}.

\bibitem[Kem78]{Kempf}
George~R. Kempf.
\newblock Instability in invariant theory.
\newblock {\em Ann. of Math. (2)}, 108(2):299--316, 1978.

\bibitem[Mum65]{GIT}
D.~Mumford.
\newblock {\em Geometric invariant theory}.
\newblock Ergebnisse der Mathematik und ihrer Grenzgebiete, Neue Folge, Band
  34. Springer-Verlag, Berlin, 1965.

\end{thebibliography}
\end{document}